\documentclass[11pt]{amsart}
\usepackage{amsmath,amsthm,amsfonts,amssymb,times}

\usepackage{titlesec}
\titleformat{\section}  
{\normalfont\Large\bfseries}
{\thesection}{1em}{}

\titleformat{\subsection}
{\normalfont\large\bfseries}
{\thesubsection}{1em}{}

\DeclareMathAlphabet{\curly}{U}{rsfs}{m}{n}  

\theoremstyle{remark}
\newtheorem{remark}{Remark}
\theoremstyle{plain}

\newtheorem{lem}{Lemma}[section]
\newtheorem{thm}{Theorem}

\numberwithin{equation}{section}

\newenvironment{romenumerate}{\begin{enumerate}
 }{\end{enumerate}}

\newcommand{\be}{\begin{equation}}
\newcommand{\ee}{\end{equation}}
\newcommand{\benn}{\begin{equation*}}
\newcommand{\eenn}{\end{equation*}}
\newcommand{\bal}{\begin{align*}}
\newcommand{\ea}{\end{align*}}
\newcommand{\eal}{\ensuremath{\end{align*}}}
\newcommand{\bea}{\begin{eqnarray}}
\newcommand{\eea}{\end{eqnarray}}

\renewcommand{\a}{\ensuremath{\alpha}}
\renewcommand{\b}{\ensuremath{\beta}}

\newcommand{\eps}{\ensuremath{\varepsilon}}

\renewcommand{\(}{\left(}
\renewcommand{\)}{\right)}

\newcommand{\pfrac}[2]{\left(\frac{#1}{#2}\right)}
\newcommand{\fl}[1]{{\ensuremath{\left\lfloor {#1} \right\rfloor}}}

\renewcommand{\le}{\leqslant}
\renewcommand{\leq}{\leqslant}
\renewcommand{\ge}{\geqslant}
\renewcommand{\geq}{\geqslant}

\begin{document}

\title{A problem of Ramanujan, Erd\H os and K\'atai on the iterated divisor function}
\author[Y. Buttkewitz]{Yvonne Buttkewitz}
\author[C. Elsholtz]{Christian Elsholtz}
\author[K. Ford]{Kevin Ford}
\author[J.-C. Schlage-Puchta]{Jan-Christoph Schlage-Puchta}

\date{\today}

\address{{\bf Y.~Buttkewitz}: Department of Mathematics\\
Royal Holloway, University of London\\
Egham, Surrey TW20 OEX, U.K.}
\email{leros@t-online.de}

\address{{\bf C.~Elsholtz}: Institut f\"{u}r Mathematik A,
Technische Universit\"{a}t Graz,
Steyrergasse 30, A-8010 Graz, Austria}
\email{elsholtz@math.tugraz.at}

\address{{\bf K.~Ford}: Department of Mathematics, 1409 West Green Street, University
of Illinois at Urbana-Champaign, Urbana, IL 61801, USA}
\email{ford@math.uiuc.edu}

\address{{\bf J.-C.~Schlage-Puchta}: Department of Mathematics,
Building S22, Ghent University, 9000 Gent, Belgium}
\email{jcsp@cage.ugent.be}

\begin{abstract}
We determine asymptotically the maximal order of $\log d(d(n))$, 
where $d(n)$ is the
 number of positive divisors of $n$.  
This solves a problem first put forth by Ramanujan in 1915.
\end{abstract}

\thanks{ The work of Y.~B. and C.~E. was supported
by the  German Research Council (DFG-Grant-Number 
BU2488/1-3). \\
K.~F. was supported by National Science Foundation
grant DMS-0901339.}
\maketitle


\section{Introduction}
Let $d(n)$ denote the number of positive divisors of an integer $n$.  The 
extreme large values of $d(n)$ were studied by Wigert 
\cite{Wigert:1907}, (see also \cite[Theorem 432]{HardyWright}).
Wigert proved that
\[
m_1(x) := \max_{n\le x} \log d(n) \sim (\log 2){\frac{\log x}{\log_2 x}}.
\]
Here $\log_k x$ denotes the $k$-th iterate of the logarithm.
The lower bound comes from considering
integers of the form $N_k=p_1\cdots p_k$, where $p_j$ denotes 
the $j$th smallest prime.  Here $d(N_k)=2^k$, while $\log N_k \sim k \log k$
by the prime number theorem.  In his  seminal 1915 paper on highly composite
numbers \cite{Ramanujan:1915}, Ramanujan gave a more precise asymptotic for
$m_1(x)$.
At the very end of his paper, Ramanujan posed the problem
of finding the extreme large values of $d(d(n))$.
By considering integers of the form 
\be\label{Rama1}
2^1\cdot 3^2\cdot 5^4\cdots p_k^{p_k-1},
\ee
Ramanujan showed that
\[
m_2(x):= \max_{n\le x} \log d(d(n)) \ge (\sqrt{2}\log 4 + o(1)) \frac{\sqrt{\log x}}{\log_2 x}.
\]

 The problem of finding the order of $m_2(x)$ has been mentioned in 
 Erd\H{o}s \cite{Erdoes}, Ivi\'{c} \cite{Ivic:1995},
and has been mentioned by Ivi\'{c} in problem sessions
in Ottawa \cite{Ivic:1999} and Oberwolfach.

Erd\H{o}s and K\'atai \cite{Erdoes:1969} showed $m_2(x) = (\log x)^{1/2}(\log_2 x)^{O(1)}$
(see (4.1) on p. 270 of \cite{Erdoes:1969}).
Twenty years later
Erd\H{o}s and Ivi\'{c} \cite{Erdoes:1989} improved the upper bound to
\[ 
m_2(x) \ll \pfrac{\log x \log_2 x}{\log_3 x}^{1/2}.
\]
Smati \cite{Smati:2005, Smati:2008} gave a further improvement
\[
m_2(x) \ll \sqrt{\log x},
\]
the best estimate known to date.
Constructions similar to Ramanujan's seem rather natural, and one might
expect that $m_2(x) \ll \frac{\sqrt{\log x}}{\log_2 x}$.
This is indeed the case, as we now show.
More precisely, we prove an asymptotic formula for $m_2(x)$ with an
error term.

\begin{thm}\label{m2}
 We have 
\[
m_2(x) =  \frac{\sqrt{\log x}}{\log_2 x} \( c + O\pfrac{\log_3 x}{\log_2 x} \),
\]
where 
\[
c =\Bigg( 8  \sum_{j=1}^\infty \log^2 (1+1/j) \Bigg)^{1/2} = 2.7959802335\ldots.
\]
\end{thm}

In particular, Theorem \ref{m2} implies that
\[
 \limsup_{n\to \infty} \frac{\log d(d(n)) \log_2 n}{\sqrt{\log n}} = c.
\]
Ramanujan's examples \eqref{Rama1} are seen to be suboptimal
with respect to the constant $c$, since $\sqrt{2}\log 4 = 1.9605\ldots$.

There is a closely related problem, to estimate the extreme values of $\omega(d(n))$, 
where $\omega(n)$ is the number of distinct prime factors of $n$. 
In fact, both  Erd\H{o}s and Ivi\'{c} \cite{Erdoes:1989}
and Smati \cite{Smati:2008} obtained upper bounds for $d(d(n))$ by first bounding
$\omega(d(n))$ and then using the elementary inequality $\log d(m)\ll (\log_2 m)
\omega(m)$ (see, e.g., Lemme 3.3 of \cite{Smati:2005} or Lemma \ref{domega} below).
 For this
problem, Ramanujan's examples \eqref{Rama1} are essentially optimal,
providing the true order and constant in
the asymptotic for $w(x) = \max_{n\le x} \omega(d(n))$.

\begin{thm}\label{wn}
We have
\[
 w(x) = \frac{\sqrt{\log x}}{\log_2 x} \( \sqrt{8} + O\pfrac{\log_3 x}{\log_2 x} \),
\] 
\end{thm}

Previously, Erd\H{o}s and Ivi\'{c} \cite{Erdoes:1989} had shown
\[
 w(x) \ll \pfrac{\log x \log_3 x}{\log_2 x}^{1/2},
\]
and later Smati \cite{Smati:2005}
found the true order $w(x) \ll \frac{\sqrt{\log x}}{\log_2 x}$.

%
\section{The lower bound in Theorem \ref{m2}}\label{lower}
%

{\bf Notation and basic prime number estimates.}
Throughout, we make use of the asymptotic
\be\label{pj}
p_j = j(\log j+ \log_2 j + O(1)),
\ee
which is a simple consequence of the prime number theorem with error term
$\pi(x)=\frac{x}{\log x}+O(\frac{x}{\log^2 x})$.  Here $\pi(x)$ is the number of
primes which are $\le x$.  We also denote by $\Omega(n)$ the number of prime power
divisors of $n$.

\begin{proof}[Proof of the lower bound in Theorem \ref{m2}]
Let $x$ be large and define $\eps = 10 \frac{\log_3 x}{\log_2 x}.$
Let 
\be\label{ai}
t=\fl{ \( \frac{8\log 2}{c}-\eps\) \frac{\sqrt{\log x}}{\log_2 x}}, 
\qquad a_i=\fl{\frac{1}{2^{i/t}-1}} \quad (1\le i\le t),
\ee
and let
\[
 n=(p_1 \cdots p_{a_1})^{p_1-1}(p_{a_1+1}\cdots p_{a_1+a_2})^{p_2-1} \cdots
(p_{a_1+\cdots+a_{t-1}+1} \cdots p_{a_1+\cdots+a_t})^{p_t-1}.
\]
The Taylor expansion of $\exp ( \frac{\log 2}{t} )$ shows that
$a_1=\lfloor (2^{1/t}-1)^{-1} \rfloor = t/\log 2+O(1)$.
By \eqref{ai}, for every positive integer $j$, 
there are $y_j :=\lfloor \frac{\log (1+1/j)}{\log 2} t \rfloor$ indices $i$
with $a_i \ge j$.  Also, $a_1+\cdots+a_t\ll t\log t$.
Using \eqref{pj}, we have $\log p_{a_1+\cdots+a_i}\le \log t +2\log_2 t+O(1)$, hence 
\[
 \log n \le \sum_{i=1}^t a_i (p_i-1) \log p_{a_1+\cdots+a_i} \le
\(\log^2 t + 3(\log_2 t)\log t + O(\log t)\)\sum_{i=1}^t ia_i.
\]
From $y_j=O(t/j)$ and the definition of $c$ we obtain
\be\label{lower1}
\begin{split}
\sum_{i=1}^t ia_i = \sum_{j\le a_1} \frac{y_j(y_j+1)}{2} &= \frac12 \sum_{j=1}^\infty 
\pfrac{\log(1+1/j)}{\log 2}^2 t^2 + O(t\log t) \\
&= \frac{c^2}{16(\log 2)^2} t^2 + O(t\log t).
\end{split}\ee
From the definition of $t$, $\log t = \frac12\log_2x-\log_3 x+O(1)$ and $\log_2 t=\log_3 x + O(1)$.
Thus,
\[
 \log n \le \(1 + \frac{2\log_3 x+O(1)}{\log_2 x} \) \(1- \frac{c\eps}{8\log 2}\)^2 \log x.
\]
Hence, if $x$ is large enough, then $n\le x$.  From the definition of $n$ above,
we have $d(n)=p_1^{a_1} \cdots p_t^{a_t}$.  Therefore,
\be\label{lower2}
\begin{split}
\log m_2(x) \ge \log d(d(n)) &= \sum_{i=1}^t \log(a_i+1) 
=\sum_{j \geq 1}(y_j-y_{j+1})\log(j+1)
= \sum_{j\ge 1} y_j \log (1+1/j) \\
&= \sum_{j\le a_1}\( \frac{\log^2 (1+1/j)}{\log 2}t + O(1/j)\) \\
&= \frac{c^2}{8\log 2} t + O(\log t) \\
&= \frac{\sqrt{\log x}}{\log_2 x} \( c + O\pfrac{\log_3 x}{\log_2 x} \).
\end{split}
\ee
\end{proof}

%
\section{Proof of the upper bound in Theorem \ref{m2}}\label{upper}
%

\begin{lem}\label{2.2'}
 Let $m_N = \min\{  m: d(m)=N \}$ and write $m_N=p_1^{\a_1} \cdots p_r^{\a_r}$.  We have
\begin{romenumerate}
 \item $\a_1 \ge \cdots \ge \a_r$,
 \item $N'|N$ implies $m_{N'}\le m_N$,
 \item for each integer $k\ge 1$, if $p_j>p_{r+1}^{1/2^k}$, then $\Omega(\a_j+1)\le k$.
\end{romenumerate}
\end{lem}

\begin{remark}
 Using \eqref{pj} and taking $k=1$, we see
from (iii) that
if $r$ is large, then $\a_j+1$ is prime for $\sqrt{r} < j\le r$.  
Also, by (iii), $\Omega(\a_j+1)\ll\log_2 r$ for all $j$.
\end{remark}

\begin{proof}
(i) This is trivial and was observed by Ramanujan \cite[(32)]{Ramanujan:1915}.  

(ii) If $N'|N$, we can
find $\a_j'\le \a_j$ for each $j$ such that $N'=(\a_1'+1)\cdots(\a_r'+1)$,
and clearly $m_{N'} \le p_1^{\a_1'}\cdots p_r^{\a_r'} \le m_N$. 

(iii) If $p_j > p_{r+1}^{1/2^k}$ and $\Omega(\a_j+1)>k$, then there are integers $a,b$ with 
$\a_j+1=ab$, $a\ge 2$ and $b\ge 2^k$.  Letting
\[
 m^*=p_j^{b-1} p_{r+1}^{a-1} \prod_{i\ne j} p_i^{\a_i},
\]
we see that $d(m^*)=d(m_N)=N$, but
\[
 \frac{m^*}{m_N} = p_j^{b-1-\a_j} p_{r+1}^{a-1} = (p_j^{-b}p_{r+1})^{a-1} < 1,
\]
a contradiction.  
\end{proof}

\begin{lem}\label{domega}
For every $\eps>0$, and for $\omega(n)=s\ge 2$ we have 
\[
 d(n) \ll_\eps \pfrac{(2+\eps)\log n}{s\log s}^s.
\]
\end{lem}

\begin{proof}
 Write the prime factorization of $n$ as $n=q_1^{a_1} \cdots q_s^{a_s}$, where
$q_1 < \cdots < q_s$.  Using the arithmetic mean - geometric mean inequality
and that $q_i\ge p_i$, we have
\[
 d(n)\le \prod_{i=1}^s (2a_i) \le 2^s\prod_{i=1}^s (a_i\log q_i) \prod_{i=1}^s
(\log p_i)^{-1}
\le \pfrac{2\log n}{s}^s \frac{(\log s)^{\pi(s)-s}}{\log 2},
\]
the last inequality coming from excluding factors corresponding to $3\le p_i<s$.
Finally, the prime number theorem implies $(\log s)^{\pi(s)} \le (\log s)^{O(s/\log s)}
\ll_\eps (1+\eps/2)^s$.
\end{proof}

{\bf Remark.} Lemma \ref{domega} is fairly sharp. For example, from the 
inequality $s=\omega(n) \le (1+o(1)) \frac{\log n}{\log_2 n}$, and the
observation that $m_1(x)$ is monotonically increasing,
we immediately obtain Wigert's upper bound for $\log d(n)$.


\medskip

The following is the key lemma, which explains the constant $c$.

\begin{lem}\label{2.4'}
Let $a_1,\ldots,a_t$ be positive integers. 

(a) we have
\[
 \sum_{i=1}^t \log(a_i+1) \le \frac{c}{2} \Bigg( \sum_{i=1}^t ia_i \Bigg)^{1/2}.
\]
Moreover, the constant $c/2$ is best possible.

(b) If $a_i\ge A$ for all $i$, where $A$ is  a positive integer, then
\[
 \sum_{i=1}^t \log(a_i+1) \le \Bigg( \frac{1+\log^2(A+1)}{A} \sum_{i=1}^t ia_i \Bigg)^{1/2}.
\]
\end{lem}

\begin{proof}
(a)  Without loss of generality, suppose $a_1 \ge \cdots \ge a_t$.  Let $y_j=\# \{i: a_i\ge j\}$.
Then
\[
 \sum_{i=1}^t ia_i = \sum_{j\ge 1} \frac{y_j(y_j+1)}{2} \ge \frac12 \sum_{j\ge 1} y_j^2.
\]
By partial summation and the Cauchy-Schwarz inequality,
\be\label{CS}
\begin{split}
 \sum_{i=1}^t \log(a_i+1) = \sum_{j\ge 1} (y_j-y_{j+1})\log(j+1)
&= \sum_{j\ge 1} y_j \log(1+1/j) \\
&\le \Bigg(\sum_{j\ge 1} y_j^2\Bigg)^{1/2} \pfrac{c^2}{8}^{1/2}.
\end{split}
\ee
Moreover, the inequality in \eqref{CS} is an equality if and only if for some real $Y$, $y_j=Y\log(1+1/j)$ for 
every $j$.  As the $y_j$ are integers, this cannot happen.  However, we can come
very close to equality in \eqref{CS} by taking $t$ large and choosing the $a_i$ by \eqref{ai},
so that $y_j=\lfloor \frac{\log(1+1/j)}{\log 2} t \rfloor$.
By \eqref{lower1} and \eqref{lower2}, we have in this case
\[
 \sum_{i=1}^t \log(a_i+1) = \frac{c^2}{8\log 2}t + O(\log t), \qquad
\sum_{i=1}^t ia_i = \frac{c^2}{16(\log 2)^2}t^2+O(t\log t),
\]
whence
\[
 \sum_{i=1}^t \log(a_i+1) = \frac{c}{2} \(1 + O\pfrac{\log t}{t} \) \(\sum_{i=1}^t ia_i
\)^{1/2}. 
\]

(b) Observe that $y_1=y_2=\cdots=y_A$.  Arguing similarly to \eqref{CS}, we obtain
\begin{align*}
\sum_{i=1}^t \log(a_i+1) &= \frac{\log(A+1)}{A}(y_1+\cdots +y_A)+\sum_{j>A} y_j \log(1+1/j) \\
&\le \Bigg(\sum_{j\ge 1} y_j^2\Bigg)^{1/2} \( A \pfrac{\log (A+1)}{A}^2 + \sum_{j>A}
\log^2(1+1/j) \)^{1/2}.
\end{align*}
Observing that $\log(1+1/j) < 1/j$ and $\sum_{j>A} 1/j^2 < 1/A$, we obtain (b).
\end{proof}

The next lemma is trivial.

\begin{lem}\label{sump}
 For any positive integer $m$, $m\ge \sum_{p|m} p$.
\end{lem}


\begin{proof}[Proof of Theorem \ref{m2}, upper bound]
Let $n$ be large, let $N=d(n)$ and factor $N=N'N''$, where
\[
 N'=u_1^{b_1}\cdots u_w^{b_w}, \qquad N''=q_1^{a_1}\cdots q_s^{a_s},
\]
where $u_1<\cdots<u_w$, $q_1<\cdots<q_s$ are primes, $b_i>(\log_2 n)^6$ for every $i$
and $a_i\le (\log_2 n)^6$ for every $i$. 

Write $m_{N'}=p_1^{\b_1} \cdots p_h^{\b_h}$.   By Lemma \ref{2.2'} (ii),
$m_{N'} \le m_N \le n$, so that $h\ll \log n$.  By Lemma \ref{2.2'} (iii),
$\Omega(\b_i+1) \ll \log_2 h \ll \log_3 n$ for every $i$.  Since $d(m_{N'})=(\beta_1+1)\cdots
(\beta_h+1)=N'$, for each $j\le h$ there
are $\gg \frac{b_j}{\log_3 n}$ values of $i$ for which $u_j|(\b_i+1)$.  Thus,
by Lemma \ref{sump},
\begin{align*}
\log n\ge \log m_{N'} \ge (\log 2)\sum_{i=1}^h \b_i &\ge \frac{\log 2}{2} \sum_{i=1}^h (\b_i+1)\\
&\ge \frac{\log 2}{2} \sum_{i=1}^h \sum_{p|(\b_i+1)} p \gg \sum_{j=1}^w u_j
\frac{b_j}{\log_3 n} \ge \frac{1}{\log_3 n} \sum_{j=1}^w j b_j.
\end{align*}
Combining this estimate with
Lemma \ref{2.4'} (b) with $A=(\log_2 n)^6$ gives
\be\label{d(N')}
\log d(N')=\sum_{j=1}^w \log(b_j+1) \ll \frac{\log_3 n}{(\log_2 n)^3} \( 
\sum_{j=1}^w jb_j \)^{1/2} \ll \frac{(\log n)^{1/2} (\log_3 n)^{3/2}}{(\log_2 n)^3}.
\ee

Next, we bound $d(N'')$.\\
  Case 1)
If $s\le \frac{(\log n)^{1/2}}{(\log_2 n)^3}$, Lemma
\ref{domega} implies that 
$\log d(N'') \ll \frac{(\log n)^{1/2}}{(\log_2 n)^2}$.\\
Case 2) Now suppose that $s > \frac{(\log n)^{1/2}}{(\log_2 n)^3}$.
Write $m_{N''}=p_1^{\a_1} \cdots p_r^{\a_r}$.
By Lemma \ref{2.2'} (iii),
\be\label{OmegaN''}
r \le \Omega(N'')=\sum_{j=1}^s a_j=\sum_{i=1}^r \Omega(\alpha_i+1)
 \le r + \sum_{k\ge 2} \pi(p_{r+1}^{1/2^k}) =r + O((r/\log r)^{1/2}).
\ee
In particular, $r + O((r/\log r)^{1/2}) \ge a_1+\cdots+a_s \ge s$, so $r\gg s> \frac{(\log n)^{1/2}}{(\log_2 n)^3}$.
Thus, for large enough $n$, $a_1+\cdots+a_s\le r +\sqrt{r}$.
Also by Lemma \ref{2.2'} (iii), $\a_j+1$ is prime for $j>\sqrt{r}$.
Let $\eps = 20 \frac{\log_3 n}{\log_2 n}$.
 By the lower bound on $s$, and using $a_i\leq (\log_2 n)^6$,
\be\label{bigl}
 \sum_{j>s-s^{1-\eps}} a_j \ge s^{1-\eps} \ge 2\(s(\log_2 n)^6\)^{1/2} \ge 
2\(\Omega(N'')\)^{1/2} \ge 2\sqrt{r},
\ee
hence, using \eqref{OmegaN''},
\[
 \sum_{j\le s-s^{1-\eps}} a_j \le \Omega(N'')-2\sqrt{r} \le r - \sqrt{r}.
\]
Using Lemma \ref{2.2'} (i), $\alpha_i+1=q_1$ for $r-a_1<i\le r$, and similarly
for each $j\le s-s^{1-\eps}$, $\alpha_i+1=q_j$ for $r-(a_1+\cdots+a_j)<i\le r-(a_1+\cdots
+a_{j-1})$.  We obtain
\begin{align*}
 \log m_{N''} &\ge \sum_{\sqrt{r}<i\le r} \a_i\log p_i \ge \sum_{j\le s-s^{1-\eps}}
(q_j-1) \sum_{m=r-(a_1+\cdots+a_j)+1}^{r-(a_1+\cdots+a_{j-1})} \log p_m \\
&\ge \sum_{j\le s-s^{1-\eps}} (p_j-1) a_j \log (r-(a_1+\cdots+a_j)).
\end{align*}
By \eqref{bigl}, uniformly for  $j\le s-s^{1-\eps}$ we have
\[
 r-(a_1+\cdots+a_j)=r-\Omega(N'')+a_{j+1}+\cdots+a_s \ge
s-j-\sqrt{r} \ge \frac12 s^{1-\eps}.
\]
Using \eqref{pj}, $p_j \ge j\log j+1$ for large $j$. Hence, by Lemma \ref{2.2'} (ii),
\begin{align*}
\log n \ge \log m_{N''} &\ge \sum_{s^{1-\eps}\le j\le s-s^{1-\eps}} (j\log j)a_j(\log s + O(\log_3 n)) \\
&\ge (1+O(\eps))\frac{(\log_2 n)^2}{4} \sum_{s^{1-\eps} \le j\le s-s^{1-\eps}} j a_j.
\end{align*}
By the definition of $\eps$, $s^{\eps} \gg (\log_2 n)^9$.
Also, trivially $\sum_{j=1}^s ja_j \ge 1+2+\cdots +s \ge \frac12 s^2$.  Recalling
that $a_j \le (\log_2 n)^6$ for every $j$, we have
\begin{align*}
\sum_{s^{1-\eps} \le j\le s-s^{1-\eps}} j a_j &=
\sum_{j=1}^s ja_j + O\( s^{2-\eps}(\log_2 n)^6\) = \sum_{j=1}^s ja_j + O(s^2 (\log_2 n)^{-3})\\
&=(1+O(1/\log_2 n)) \sum_{j=1}^s ja_j.
\end{align*}
Combining the last two inequalities gives
\[
 \log n \ge \(1+O\pfrac{\log_3 n}{\log_2 n}\) \frac{(\log_2 n)^2}{4} \sum_{j=1}^s ja_j.
\]
Applying Lemma \ref{2.4'} (a), we conclude that
\be\label{dN''}
\log d(N'') = \sum_{j=1}^s \log (a_j+1) \le \frac{c}{2} \Big(\sum_{j=1}^s ja_j \Big)^{1/2}
\le c \frac{\sqrt{\log n}}{\log_2 n}\(1 + O\pfrac{\log_3 n}{\log_2 n} \).
\ee
Recall that we have a smaller upper bound for $\log d(N'')$ in case 1).  
Finally, using $d(d(n))=d(N')d(N'')$
and combining \eqref{d(N')} and \eqref{dN''}, we obtain the desired upper bound for
$d(d(n))$.
\end{proof}


\section{Proof of Theorem~\ref{wn}}

\begin{proof}[Proof of Theorem \ref{wn}]
For the lower bound, let $x$ be large and 
put $n=\prod_{i=1}^s p_i^{p_i-1}$, where $s$ is the
largest integer such that $n\le x$.   Recall that $p_j$ is the $j$-th smallest prime.
Then
$d(n)=\prod_{i=1}^s p_i$, thus $\omega(d(n))=s$.  By \eqref{pj},
\[
\log n = \sum_{i=1}^s (p_i-1)\log p_i = \sum_{i=1}^s i\log^2 i + O(i\log i\log_2 i)
= \frac{1}{2}s^2\log^2 s + O(s^2\log s\log_2 s).
\]
Solving for $s$ gives $s = \frac{\sqrt{8\log n}}{\log_2 n}+O(\frac{\sqrt{\log n}\log_3 n}{\log_2^2 n})$.
We now prove a lower bound on $n$.  Since $p_{s+1} \sim s\log s\sim \sqrt{2\log n} \ll
\sqrt{\log x}$ by \eqref{pj}, we have
\[
 x \ge n \ge x p_{s+1}^{-p_{s+1}} = x \exp \( - O\( \sqrt{\log x}\log_2 x \)\).
\]
That is, $\log n = \log x + O(\sqrt{\log x}\log_2 x)$.
Therefore, $s = \frac{\sqrt{8\log x}}{\log_2 x}+O(\frac{\sqrt{\log x}\log_3 x}{\log_2^2 x})$.

Now let $n$ be a large, positive integer factored as
$n = n_1 n_2$, $n_1=\prod_{i=1}^r q_i^{a_i}$, $n_2= \prod_{i=1}^{r'} (q_i')^{a_i'}$,
where $q_i,q_i'$ are primes, $q_i>P$ and $q_i'\le P$ for each $i$, where
$P=\frac{\sqrt{\log n}}{\log_2 n}$.  We have 
\be\label{omega12}
\omega(d(n)) \le \omega(d(n_1)) + \omega(d(n_2)).            
\ee
Since $\omega(n_2)\le \pi(P) \ll \frac{\sqrt{\log n}}{(\log_2 n)^2}$, 
Lemma \ref{domega} implies $\log d(n_2) \ll \sqrt{\log n}/\log_2 n$.
Applying the elementary inequality $\omega(u)\ll \frac{\log u}{\log_2 u}$ gives
\be\label{wn2}
\omega(d(n_2)) \ll \frac{\sqrt{\log n}}{(\log_2 n)^2}.
\ee
Next,
\[
\log n_1 \geq (\log P) \sum_{i=1}^r a_i =
\( \frac{\log_2 n}{2} - \log_3 n \) \sum_{i=1}^r a_i.
\]
Letting $s=\omega(d(n_1))=\omega(\prod (a_i+1))$, Lemma \ref{sump}
implies that
\[
\sum_{i=1}^r a_i \ge \sum_{i=1}^r \sum_{p|(a_i+1)} (p-1) \ge \sum_{i=1}^{s} (p_i-1) \ge
\sum_{i=1}^{s} (i\log i + O(1))  = \frac{1}{2}{s}^2\log s + O(s^2).
\]
Here we used the one-sided inequality $p_i \ge i\log i + O(1)$ deduced from \eqref{pj}.
Thus,
\[
\log n \ge \log n_1 \ge \(\frac{1}{4}+O\pfrac{\log_3 n}{\log_2 n}\)(\log_2 n) {s}^2\log s
 + O({s}^2\log_2 n).
\]
Consider two cases: (i) $s \le  \frac{\sqrt{\log n}}{\log_2 n}$, (ii)
$s > \frac{\sqrt{\log n}}{\log_2 n}$.  In case (ii), we have
$\frac{\log n}{\log_2^2 n} \geq (\frac{1}{8}+O(\frac{\log_3
  n}{\log_2 n})){s}^2$, and we obtain in both cases
\[
\omega(d(n_1)) = s \le \frac{\sqrt{8\log n}}{\log_2 n} + 
O\pfrac{\sqrt{\log n}\log_3 n}{\log_2^2 n},
\]
Combining this inequality with \eqref{omega12} and \eqref{wn2}, 
we obtain the desired upper bound for $\omega(d(n))$.
\end{proof}

\medskip

{\bf Acknowledgement.}
The authors would like to thank Prof. L.~G.~Lucht for discussions
on the subject and helpful comments on an earlier version of this paper.


\end{document}